\documentclass[11pt,oneside]{amsart}
  
\usepackage{geometry} 
\usepackage{amsmath, amsthm, amssymb} 
\usepackage{hyperref}
\usepackage{lipsum} 
\usepackage{scrextend}
\hypersetup{colorlinks=true,linkcolor=blue, linktocpage}
\usepackage{graphicx} 
\usepackage{mathrsfs} 
\usepackage{color} 
\usepackage{verbatim}
\usepackage[bbgreekl]{mathbbol}
\usepackage[dvipsnames]{xcolor}
\usepackage{enumitem}
\usepackage{graphicx}

\usepackage{geometry}
\makeatletter
\@addtoreset{equation}{section}
\makeatother

\usepackage[all]{xy}
\usepackage{hyperref}
\title{Bounding non-rationality of divisors on 3-fold {Fano fibrations}}
\author{Caucher Birkar and Konstantin Loginov}
\date{} 

\newcounter{cthm}

\newtheorem{proposition}[equation]{Proposition}
\newtheorem{thm}[equation]{Theorem}

\newtheorem{lem}[equation]{Lemma}
\theoremstyle{definition}

\newtheorem{remark}[equation]{Remark}
\newtheorem{question}[equation]{Question}

\newtheorem{exam}[equation]{Example}
\theoremstyle{exam}

\newcommand{\OOO}{\mathscr{O}}

 \newcommand{\PP}{\mathbb P}
 
 \newcommand{\Q}{\mathbb Q}
 \newcommand{\R}{\mathbb R}

 \newcommand{\bir}{\dashrightarrow}
 \newcommand{\rddown}[1]{\left\lfloor{#1}\right\rfloor} 

\newcommand{\Addresses}{{
  \bigskip
  \footnotesize
  
  \  
  
  \textsc{Yau Mathematical Sciences Center,  
  Tsinghua University, Hai Dian District, Beijing, China, Post Code: 100084 }\\
  \textit{E-mail:} \texttt{c.birkar@dpmms.cam.ac.uk}
  
  \
  
  \textsc{Steklov Mathematical Institute of Russian Academy of Sciences, Moscow, Russia; Laboratory of Algebraic Geometry, HSE University, Russian Federation; Laboratory of AGHA, Moscow Institute of Physics and Technology. }\\
  \textit{E-mail:} \texttt{kostyaloginov@gmail.com}
}}
\begin{document}

\maketitle

\begin{abstract}
In this paper we investigate non-rationality of divisors on 3-fold log Fano fibrations $(X,B)\to Z$ under mild conditions. We show that if $D$ is a component of $B$ with coefficient $\ge t>0$ which is contracted to a point on $Z$, then $D$ is birational to $\mathbb{P}^1\times C$ where $C$ is a smooth projective curve with gonality bounded depending only on $t$. Moreover, if $t>\frac{1}{2}$, then genus of $C$ is bounded depending only on $t$.
\end{abstract}


\section{Introduction}

We work over an algebraically closed field $\mathbb{k}$ of characteristic $0$. It is well-known that a smooth del Pezzo surface, which is automatically rational, 
can degenerate into a non-rational singular del Pezzo surface. Indeed, we can pick a smooth hypersurface in $\mathbb{P}^3$ of 
degree $3$ which is a smooth del Pezzo surface degenerating into a cone over an elliptic curve which is a non-rational singular 
del Pezzo. This happens because the degeneration has log canonical (lc) 
but not Kawamata log terminal (klt) singularities.  

The aim of this paper is to investigate the above phenomenon in greater depth. More precisely, 
we consider log Fano fibrations from 3-folds and 
try to understand how far the components of the special fibres can be from being rational. Unlike the example 
above we allow the special fibres to be non-reduced and reducible.  

Let $f\colon X\to Z$ be a klt Fano fibration, that is, $f$ is a contraction, 
$X$ has klt singularities, and $-K_X$ is ample over $Z$. 
By contraction we mean a projective morphism with connected fibres. 
Let $F$ be a fibre of $f$ (over a closed point) and 
let its reduction be $F_{\rm red}$ with irreducible components $F_i$. 
Then $F_{\rm red}$ is rationally chain connected; moreover, if $F$ is a general fibre, then $F=F_{\rm red}$ is 
rationally connected \cite{Zh06}\cite{HM07}. We then focus on the special fibres and 
would like to understand how far the components $F_i$ are from being rational. 
Since $F_{\rm red}$ is rationally chain connected, $F_i$ are uniruled.

In this paper we treat the case when $X$ is a 3-fold and $\dim Z>0$. So $f$ is either birational or a 
conic bundle or a del Pezzo fibration. If $F_i$ is a point, then there is nothing to say. If 
$F_i$ is a curve, then it is a rational curve. So we naturally focus on the case when $F_i$ is a 
surface in which case $F_i$ is birational to $\mathbb{P}^1\times C_i$ 
for some smooth projective curve $C_i$. The question is then how far $C_i$ is from being rational. 
One way to measure this is to ask how large the genus of $C_i$ is. Another way is to ask how large its gonality is.  

{In the simplest case, that is, when $X$ is smooth (or more generally terminal Gorenstein) and $X\to Z$ is a del Pezzo fibration  and $F$ is irreducible, one can easily show  that $F$, which is automatically reduced in this case, is either rational or isomorphic to a {generalised} cone over an elliptic curve (cf. \cite[1.3]{Lo19}). In other words, $F$ is birational to $\mathbb{P}^1\times C$ where $C$ is either a smooth rational curve or an elliptic curve. Moreover, one can describe the local structure of such fibrations in the neighbourhood of non-rational fibers. It turns out that germs of such fibrations with very mild singularities and non-rational irreducible special fibre are in one-to-one correspondence with germs of smooth del Pezzo fibrations that admit an action of a cyclic group, see \cite[2.4, 3.4]{Lo19}. This shows that the non-rational fibers of terminal Gorenstein del Pezzo fibrations form a very restricted class. On the other hand, if we allow $X$ to have arbitrary klt singularities then both the genus and the gonality of $C$ may be arbitrarily large, see Example \ref{exa-role-of-t}.} 

{A related problem was treated in \cite{BCDP19}. There the authors considered groups of birational self-maps of 
a fixed projective threefold $X$. More precisely, they studied the birational type of surfaces contracted by these self-maps. Such surfaces are birational to $\mathbb{P}^1\times C$. It is shown that the genus of $C$ is unbounded if and only if $X$ is birational to a conic bundle or a fibration whose generic fibre is a del Pezzo surface of degree $3$. Similarly, the gonality of the curve is unbounded if and only if $X$ is birational to a conic bundle. The following question is proposed in \cite[4.6]{BCDP19}: is there an integer $g \geq 1$ such that for each 3-fold Mori fibre space $X\to Z$ onto a smooth curve, 
a reduced fibre is birational to $\mathbb{P}^1\times C$ for some curve $C$ of genus (respectively gonality) $\leq g$? 
Here the authors assume $X$ has at most terminal $\mathbb{Q}$-factorial singularities (in this paper we work 
with Mori fibre spaces with klt singularities, see \ref{def-MFS}). 

There are also other ways to measure how far a given algebraic variety is from being rational, see \cite{BPELU17}.}\\

Our first general result concerns the non-rationality of fibres in del Pezzo fibrations.

\begin{thm}
\label{the_theorem}
Fix a positive real number $t$. Assume that $f\colon X\to Z$ is a klt Fano fibration  
where $\dim X=3$ and $\dim Z=1$. Assume $F$ is an irreducible fibre and that $(X, tF_{\rm red})$ is lc. Then 
\begin{enumerate}
\item  $F_{\rm red}$ is birational to $\mathbb{P}^1\times C$ where $C$ is a smooth projective curve 
with gonality ${\mathrm{gon}(C)}$ bounded depending only on $t$;

\item  if $t>\frac{1}{2}$, then  the genus $g(C)$ is bounded;

\item  if $t=1$, then the genus $g(C)\leq 1$.
\end{enumerate}
\end{thm}

We will see below in Example \ref{exa-role-of-t} that without the assumption 
that $(X, tF_{\rm red})$ is lc, both $g(C)$ and ${\mathrm{gon}(C)}$ can be arbitrarily large.  
In Example \ref{exa-role-of-t2} we see that even assuming $(X, tF_{\rm red})$ is lc, 
$g(C)$ can be arbitrarily large when $t\le \frac{1}{2}$. 

Theorem \ref{the_theorem} is expected to be especially useful on Fano fibrations with good singularities. 
Indeed if $X\to Z$ is a Fano fibration from a 3-fold onto a smooth curve where $X$ has 
$\epsilon$-lc singularities with $\epsilon>0$ (e.g., terminal or canonical singularities), 
then a conjecture of Shokurov (cf. \cite{B16}) predicts that $(X,tF)$ is lc for any fibre $F$ and for some 
$t>0$ depending only on $\epsilon$. In particular, if $F$ is irreducible, then the above theorem 
says that in this case (assuming Shokurov's conjecture) $F_{\rm red}$ is birational to 
$\PP^1\times C$ where $C$ is a smooth projective curve with gonality bounded depending only on $\epsilon$.
In particular, Shokurov's conjecture implies that the question of \cite{BCDP19} mentioned above 
has an affirmative answer in the sense that the relevant gonality is bounded since in their setting 
the total space has terminal singularities.

Case (iii) of Theorem \ref{the_theorem} is much simpler than the other cases. We outline the proof here. 
By assumption, $t=1$ and that $(X,F_{\rm red})$ is lc. 
Let $S$ be the normalization of $F_{\rm red}$. By adjunction we can write   
$K_{S}+\Delta_{S}= (K_X+F_{\rm red})|_{S}$ where $(S,\Delta_S)$ is an lc log Fano pair because 
$-K_X$ is ample over $Z$ and $F_{\rm red}$ is numerically trivial over $Z$ as $F$ is irreducible. 
If $({S},\Delta_S)$ is klt, then $S$ is rational by \cite{Zh06}. 
Hence we may assume that this pair is not klt. We have 
\[
K_S+\Delta_S + H\sim_{{\mathbb{Q}}} 0
\]
for some ample $\mathbb{Q}$-divisor $H$ such that $(S,\Delta_S+H)$ is lc. 
Thus $(S, \Delta_S + H)$ is an lc log Calabi-Yau surface {pair}.
The rest of the proof follows from taking a dlt model and running MMP, 
see Lemma \ref{lc_log_cy_bounded}. 

Theorem \ref{the_theorem} is a consequence of the next far more general 
statement in which we allow the fibres to be non-reduced and reducible, $Z$ can be a curve, a surface, 
or a 3-fold, and we work with Fano pairs. 

A \emph{dlt Fano fibration} $f\colon (X,B)\to Z$ consists of a dlt pair $(X,B)$ and a contraction $f\colon X\to Z$ 
such that $-(K_X+B)$ is ample over $Z$.

\begin{thm}
\label{the_theorem-general}
{Fix a positive real number $t$.  Assume that $f\colon  (X,B)\to Z$ is a dlt Fano fibration where $\dim X=3$ and 
$\dim Z\ge 1$. Assume $D$ is a component of $B$ with coefficient $\ge t$ contracted to a point on $Z$.} 
Then 
\begin{enumerate}
\item $D$ is  birational to $\mathbb{P}^1\times C$ where $C$ is a smooth projective curve 
with gonality ${\mathrm{gon}(C)}$ bounded depending only on $t$; 

\item if $t>\frac{1}{2}$, then the genus $g(C)$ is bounded depending only on $t$; 

\item if $t=1$, then $g(C)\le 1$.
\end{enumerate}

\end{thm}

We now reformulate the theorem in the language of log Calabi-Yau fibrations (as in \cite{B18}) 
which is more convenient for proofs 
and puts the problems discussed in this paper in a more general context.

\begin{thm}
\label{the_theorem-general-lcy}
{Fix a positive real number $t$.  Assume that $(X,B)\to Z$ is a Fano type log Calabi-Yau fibration where $\dim X=3$ and 
$\dim Z\ge 1$. Assume $D$ is a component of $B$ with coefficient $\ge t$ contracted to a point on $Z$.} 
Then 
\begin{enumerate}
\item $D$ is  birational to $\mathbb{P}^1\times C$ where $C$ is a smooth projective curve 
with gonality ${\mathrm{gon}(C)}$ bounded depending only on $t$; 

\item if $t>\frac{1}{2}$, then the genus $g(C)$ is bounded depending only on $t$; 

\item if $t=1$, then $g(C)\le 1$.
\end{enumerate}

\end{thm}

The Fano type property is important, indeed, taking a smooth family of 
K3 surfaces over a curve one sees that already (i) fails since the fibres in such a family are not uniruled.
However, even without the Fano type property we expect some nice behaviour of non-rationality, see 
below for some brief discussions.

An interesting special case of the theorem is when one starts with a klt pair $(Z,B_Z)$ of dimension $3$ and 
asks what kind of divisors appear on resolutions of singularities of the pair. More precisely, 
suppose $\phi\colon W\to Z$ is a log resolution and write $K_W+B_W=\phi^*(K_Z+B_Z)$. 
Let $D$ be a prime divisor on $W$ contracted to a point in $Z$. One can show that $D$ is uniruled, 
e.g. by running some MMP over $Z$, so $D$ is birational to $\PP^1\times C$ where $C$ is a 
smooth projective curve. The question as before is whether the gonality of $C$ (or its genus) is 
bounded. Theorem \ref{the_theorem-general-lcy} says that the gonality is bounded if 
the coefficient of $D$ in $B_W$ is $\ge t$ for some fixed $t>0$ and that the genus is bounded 
if the coefficient is $>\frac{1}{2}$. Indeed running a suitable MMP we get a model $X\to Z$ 
on which $B \ge 0$, where $B$ is the pushdown of $B_W$, and $D$ is not contracted over $X$ 
so that we can apply the theorem to $(X,B)\to Z$ which is a Fano type log Calabi-Yau fibration.
On the other hand, if the coefficient of $D$ in $B_W$ is too small or negative, then 
in general the gonality of $C$ is not bounded.

A case that we have not treated in this paper is when we have a global Fano type 3-fold. That is, suppose 
$(X,B)$ is a projective Fano type log Calabi-Yau pair of dimension 3 and assume $D$ is a component of $B$ 
with coefficient $\ge t$ where $t>0$ is fixed. Unlike in the relative case discussed above, 
$D$ does not need to be uniruled. But this is not the end of the story. For example, if 
$X=\PP^3$ or any Fano variety in a bounded family, then $D$ belongs to a bounded family. {In particular, the degree of irrationality of $D$ is bounded, where the degree of irrationality of a variety $D$ is defined as the least possible degree of dominant rational maps $D\bir \PP^{\dim D}$. This notion generalises the notion of gonality to higher dimensions. The following general question puts this into perspective. 

\begin{question}
\emph{
Fix a positive real number $t>0$ and natural number $d$. 
Suppose that $f\colon (X,B)\to Z$ is a Fano type log Calabi-Yau fibration where $\dim X=d$. 
Assume $D$ is a component of $B$ with coefficient $\ge t$ contracted to a point on $Z$. Is is true that  
there is a rational map $D\bir C$ where the general fibres are rationally connected 
and $C$ is a smooth projective variety with bounded {degree of irrationality}?} 
\end{question}  
Without the Fano type assumption the answer to the question is expected to be negative. 
For example, it is conjectured that {degree of irrationality} of K3 surfaces $F$ are not bounded, see Conjecture 4.2
of \cite{BPELU17}, but they appear as fibres of the log Calabi-Yau fibration $F\times Z\to Z$ 
where the morphism is projection and $Z$ is a smooth curve.  

On the other hand, one may change the question and replace {degree of irrationality} of $C$ with its covering gonality which is defined to be the least gonality of covering families of curves on $C$. The covering gonality of K3 and 
abelian surfaces are $2$ (see \cite{BPELU17} and the references therein).

In the question, even if $D$ is not contracted to a point, one expects some nice properties 
of the general fibres of $D\to f(D)$. A special case of this plays a key role in our proofs, see 
\ref{plt_bounded}.


Finally, the next result is the analogue of Theorem \ref{the_theorem-general-lcy} in dimension $2$. Here we also allow 
the case $\dim Z=0$ and do not assume Fano type property.

\begin{thm}\label{t-general-lcy-dim-2}
Let $u$ be a positive real number. Let $(S, \Delta)\to Z$ be a log Calabi-Yau fibration where $S$ is of dimension $2$.
Assume that the coefficient of a component $D$ of $\Delta$ is $\ge u$ and that $D$ is contracted 
to a point on $Z$. Then
\begin{enumerate}
\item 
the gonality of $D$ is bounded depending only on $u$,
\item
{if $u>\frac{1}{2}$, then the genus $g(D^{\nu})$ of the normalization $D^{\nu}$ of $D$ is bounded depending only on $u$,

\item if $u=1$, then $g(D^{\nu}) \leq 1$,

\item if $\dim Z\ge 1$, then $g(D^{\nu}) \leq 1$.}
\end{enumerate}
\end{thm} 

{
 We will also give sharp bounds for the gonality and genus in the theorem;  
more precisely, we show that $\mathrm{gon}(D) \leq 3/u - 1$, and in case $u>1/2$, we show that 
$g(D^\nu)\leq \frac{1}{u-1/2} - 1$.
For more details, see Proposition \ref{p-dim2-effective-bnd}.}

\subsection*{Acknowledgements}
The first author was supported by a grant of the Royal Society and did this work at the University of Cambridge. {The second author was partially supported by the HSE University Basic Research Program, Russian Academic Excellence Project ``5-100'', Foundation for the Advancement of Theoretical Physics and Mathematics ``BASIS'', and the Simons Foundation. {The work on this paper was initiated while the second author visited the University of Cambridge. He thanks it for hospitality.} He also thanks Artem Avilov and Yuri Prokhorov for useful discussions, Ivan Cheltsov and {Alexei Gorinov} for making his visit to the UK possible.}


\section{Examples}

In this section we present several examples to show what happens if we drop some of the conditions of the main 
results.

\begin{exam}\label{exa-role-of-t}
We present an example to show that Theorem \ref{the_theorem} fails without the assumption 
that $(X, tF_{\mathrm{red}})$ is lc (and similarly Theorem \ref{the_theorem-general} fails without assuming the 
coefficient of $D$ in $B$ is $\ge t$). Let $X'=\mathbb{P}^2\times Z\to Z$ where 
$Z$ is a smooth curve and the morphism is projection. Let $F'$ be a fibre, say over a 
closed point $z$. 
Take a general ample divisor $H'\ge 0$ on $X'$ not containing $F'$ so that $H'|_{F'}$ is a smooth curve $C'$. 
Consider $B'=aF'+bH'$ where $a<1$ is close to $1$ but $b$ is small so that 
$a+b>1$, $(X',B')$ is klt, and $-(K_{X'}+B')$ is ample over $Z$. 
The larger $\deg C'$, the smaller $b$ has to be to ensure $-(K_{X'}+B')$ is ample over $Z$. 
Assume $X''\to X'$ is the blowup of $X'$ along $C'$ 
and $F''$ the exceptional divisor. 
Let $K_{X''}+B''$ be the pullback of $K_{X'}+B'$. Then the coefficient of $F''$ in $B''$ is 
$a+b-1$, hence it is positive. Thus $X''$ is of Fano type over $Z$ with relative Picard number $2$, 
hence running an MMP on $-F''$ contracts the birational transform of $F'$ and gives 
a model $X/Z$ of relative Picard number $1$. Therefore, $X\to Z$ is a klt Fano fibration as 
$-K_X$ is ample on the general fibres of $X\to Z$. The klt property follows from the fact that 
there is $\Delta'\ge B'$ such that $(X',\Delta')$ is klt and $K_{X'}+\Delta'\equiv 0/Z$.
Let $B,F$ be the pushdown of $B'',F''$, respectively. 
Then the fibre of $X\to Z$ over $z$ is just $F$. By construction, $F$ is birational to $F''$ which is in turn birational to 
$\PP^1\times C'$. However, {gonality (and hence genus)} of $C'$ is not bounded as $\deg C'$ is not bounded. 
Note that the larger $\deg C'$, the smaller $b$ so the smaller is the coefficient of $F$ in $B$.
\end{exam}

\begin{exam}\label{exa-role-of-t3}
We can construct similar examples with $X\to Z$ birational. Indeed pick a smooth $3$-fold $Z$, 
let $X'\to Z$ be the blowup at a smooth point with exceptional divisor $F'$, then blowup a curve 
in $F'$, and continue as in the previous example by contracting the birational transform of $F'$, etc.
\end{exam}

{\begin{exam}\label{exa-role-of-t2}
This example shows that if in Theorem \ref{the_theorem} (and similarly in 
\ref{the_theorem-general-lcy} which is a reformulation of \ref{the_theorem-general}) 
we take $t\le \frac{1}{2}$, then in general we cannot bound the genus 
of $C$ although gonality would be bounded. Consider the klt Fano fibration 
$$
X' = \mathbb{P}(1,1,n) \times Z \to Z
$$ 
where $Z$ is a smooth curve, $\mathbb{P}(1,1,n)$ is a weighted projective space, 
the morphism is projection, and $n\geq 3$. Pick a closed point $z\in Z$ and let $F'$ be the fibre over $z$.   
There exists a hyperelliptic curve $C'$ of genus $g=n-1$  
sitting in the smooth locus of the fiber $F'$ such that $-(K_{F'}+\frac{1}{2}C')$ is ample 
(see Example \ref{exam-u-greater-one-half}). 
Let $H'$ be the pullback of $C'$ under the projection $X'\to F'$. Then 
$$
(X',B':=F'+\frac{1}{2}H')
$$ 
is plt (by applying inversion of adjunction by restricting to $F'$) and $-(K_{X'}+B')$ is ample over $Z$.  
Let $X''\to X'$ be the blowup of $X'$ along $C'$ with exceptional divisor $F''$. 
Let $K_{X''}+B''$ be the pullback of $K_{X'}+B'$. 
Then the coefficient of $F''$ in $B''$ is $\frac{1}{2}$. 
Thus $X''$ is a Fano type variety over $Z$ whose fiber over $z$ consists of the exceptional divisor $F''$ plus the birational transform of $F'$. In particular, we can contract the birational transform of $F'$ by running an MMP on $-F''$. 
We obtain a model $X/Z$ so that $(X, 1/2F)$ is lc where $F$ is the pushdown of $F''$; the lc property can be 
seen similar to the previous example by taking an appropriate $\Delta'\ge B'$. Then $X\to Z$ is a klt Fano fibration.
By construction, $F$ is birational to $\mathbb{P}^1\times C'$ where $C'$ is a hyperelliptic curve of genus $g=n-1$ which can be arbitrarily large. 
\end{exam}}


\section{Preliminaries}
We work over an algebraically closed field $\mathbb{k}$ of characteristic $0$. All the varieties in this paper are quasi-projective over $\mathbb{k}$ unless stated otherwise.
\subsection{Contractions} By a contraction we mean a projective morphism $f \colon X \to Y$ of varieties such that $f_*\OOO_X = \OOO_Y$. In particular, $f$ is surjective and has connected fibres. {Moreover, if X is normal, then Y is also normal.
}
{\subsection{Divisors} Let $X$ be a variety. If $D$ is a prime divisor on birational models of $X$ whose centre on $X$ is non-empty, then we say $D$ is a prime divisor over $X$. If $X$ is normal and $M$ is an $\mathbb{R}$-divisor on $X$, we let
\[
| M |_{\mathbb{R}} = \{ N \geq 0\;|\;N \sim_{\mathbb{R}} M\}.
\]
Recall that $N \sim_{\mathbb{R}} M$ means that $N - M = \sum r_i \mathrm{Div}\,(\alpha_i)$ for certain real numbers $r_i$ and rational functions $\alpha_i$. When all the $r_i$ can be chosen to be rational numbers, then we write $N \sim_{\mathbb{Q}} M$. We define $| M |_{\mathbb{Q}}$ similarly by replacing $\sim_{\mathbb{R}}$ with $\sim_{\mathbb{Q}}$.
}
\subsection{Pairs and singularities} A \emph{pair} $(X, B)$ consists of a normal variety $X$ and a boundary $\mathbb{R}$-divisor $B$ with coefficients in $[0, 1]$ such that $K_X + B$ is $\mathbb{R}$-Cartier. {A \emph{pair} $(X, B)$ \emph{over $Z$} is a pair $(X, B)$ together with a projective morphism $X\to Z$}. Let $\phi\colon W \to X$ be a log resolution of $(X,B)$ and let $K_W +B_W = \phi^*(K_X +B)$.
The \emph{log discrepancy} of a prime divisor $D$ on $W$ with respect to $(X, B)$ is $1 - \mu_D B_W$ and it is denoted by $a(D, X, B)$, {where $\mu_D B_W$ is the multiplicity of $B_W$ along $D$}. We say $(X, B)$ is lc (resp. klt) (resp. $\epsilon$-lc) if $a(D, X, B)$ is $\geq 0$ (resp. $> 0$)(resp. $\geq \epsilon$) for every $D$. Note that if $(X, B)$ is $\epsilon$-lc, 
then automatically $\epsilon \leq 1$ because $a(D, X, B) = 1$ for almost all $D$.

A \emph{non-klt place} of $(X, B)$ is a prime divisor $D$ over $X$, that is, on birational models of $X$, such that $a(D, X, B) \leq 0$. A \emph{non-klt centre} is the image on $X$ of a non-klt place.

\subsection{Minimal model program (MMP)}
We will use standard results of the minimal model program (cf. \cite{KM98}\cite{BCHM10}). Assume $(X, B)$ is a pair over $Z$. Assume $H$ is an ample$/Z$ $\mathbb{R}$-divisor and that $K_X + B + H$ is nef$/Z$. Suppose $(X, B)$ is klt or that it is $\mathbb{Q}$-factorial dlt. Then we can run an MMP$/Z$ on $K_X + B$ with scaling of $H$. If $(X, B)$ is klt and if either $K_X + B$ or $B$ is big$/Z$, then the MMP terminates \cite{BCHM10}. In this paper we mainly work in dimension 3 in which case 
termination is known in full generality for lc pairs.

\subsection{Log Fano and log Calabi-Yau pairs} 
Let $(X, B)$ be an lc pair over $Z$. We say $(X,B)$ is \emph{log Fano over $Z$} if $-(K_X+B)$ is ample over $Z$. If $Z$ is a point then $(X, B)$ is called a \emph{log Fano pair}. In this case if $B=0$ then $X$ is called a \emph{Fano variety}. 
We say $(X,B)$ is \emph{log Calabi-Yau over $Z$} if $K_X + B \sim_{\mathbb{R}} 0$ over $Z$. If $Z$ is a point then $(X, B)$ is called a \emph{log Calabi-Yau pair}.

{\subsection{Fano type varieties} 
A variety $X$ over $Z$ is called \emph{of Fano type over $Z$} if there is a boundary $B$ such that $(X, B)$ is a klt log Fano pair over $Z$. This is equivalent to having a boundary $B$ such that $(X, B)$ is a klt log Calabi-Yau pair over $Z$ and $B$ is big over $Z$. By \cite{BCHM10}, we can run MMP over $Z$ on any $\mathbb{R}$-Cartier $\mathbb{R}$-divisor $M$ on $X$ and the MMP ends with a minimal model or a Mori fibre space for $M$.}

{\subsection{Fano fibrations and Mori fibre spaces}
\label{def-MFS}
A \emph{log Fano fibration} $(X,B)\to Z$ is a log Fano pair $(X,B)$ over $Z$ such that 
the underlying morphism $X\to Z$ is a contraction. 
A \emph{klt (resp. dlt) log Fano fibration} $\pi\colon(X,B)\to Z$ is a log Fano fibration such that 
$(X,B)$ is klt (resp. dlt). 

A log Fano fibration $(X, B)\to Z$ is called a \emph{Mori fibre space} if
\begin{enumerate}
\item $X$ is $\mathbb{Q}$-factorial,
\item $\dim X>\dim Z$, and 
\item the relative Picard number $\rho(X/Z)$ is equal to $1$.
\end{enumerate}

{\subsection{Log Calabi-Yau fibrations}
\label{def-MFS2}
A \emph{log Calabi-Yau fibration} $(X,B)\to Z$ is an lc pair $(X,B)$ over $Z$ such that 
the underlying morphism $X\to Z$ is a contraction and $K_X+B\sim_\R 0/Z$. 
If in addition $X$ is of Fano type over $X$, then we say $(X,B)\to Z$ is a Fano type log Calabi-Yau 
fibration.

\subsection{Gonality of curves}
By \emph{gonality}  of a projective curve $C$, denoted {$\mathrm{gon}(C)$}, we mean the smallest natural number $r$ such that the normalization 
$C^\nu$ admits a finite morphism $C^\nu \to \mathbb{P}^1$ of degree $r$. By definition, gonality is a birational invariant of projective curves.  To give an example, it is well-known that the gonality of a smooth plane curve $C$ of degree $d$ is equal to $d-1$ for $d\geq 2$ and equal to $1$ when $d=1$. If $C$ is singular, the 
gonality is bounded from above by $d-1$: in this case $d\ge 2$; 
blowing up a general point on $C$ and considering the resulting $\mathbb{P}^1$-bundle $S\to \mathbb{P}^1$ 
we see that the birational transform of $C$ intersects the general fibres of this fibration 
at $d-1$ points, so the gonality is $\le d-1$.}


\section{Dimension two}

In this section we prove variants of our main results in dimension two. They are interesting on their own 
but also needed to treat the three-dimensional case.

\begin{lem}
\label{lc_log_cy_bounded}
Let $(S, \Delta)$ be a log Calabi-Yau pair of dimension $2$ which is not canonical, that is, it is not $1$-lc. 
Then $S$ is birational to $\mathbb{P}^1\times C$ where $C$ is a smooth projective curve 
with genus $g(C)=0$ or $1$. 
\end{lem}
\begin{proof}
Replace $S$ with its minimal resolution and replace $K_S+\Delta$ with its pullback.  
Since $(S, \Delta)$  is not canonical, $\Delta\neq 0$. Run an MMP on $K_S$ which ends with a Mori fibre space.
Replace $S$ with the resulting model so that we can assume that we have  
a Mori fibre space structure $h\colon S\to Z$ where $\Delta$ is still not zero. 
If $Z$ is a point, then $S$ is a smooth Fano, hence rationally connected so we let $C=\mathbb{P}^1$. 
We can then assume $Z$ is a curve. 

Consider the canonical bundle formula 
$$
K_S+\Delta\sim_\R h^*(K_Z+\Delta_Z+M_Z)
$$
where $\Delta_Z,M_Z$ are the discriminant and moduli divisors; $\Delta_Z$ is effective and 
$M_Z$ is pseudo-effective (cf. Theorem 3.6 of \cite{B19}). Since 
$$
\deg (K_Z+\Delta_Z+M_Z)=0,
$$ 
we deduce that $Z$ is either $\PP^1$ or an elliptic curve. Thus in this case taking $C=Z$, the 
claim follows.

\end{proof}

\begin{proof}(of Theorem \ref{t-general-lcy-dim-2})
\emph{Step 1}.
Replacing $S$ with the minimal resolution we can assume $S$ is smooth.
The assertion (iii) immediately follows by the adjunction formula 
$$
K_{D^{\nu}}+\Delta_{D^{\nu}}:=(K_S + \Delta)|_{D^{\nu}},
$$ 
where $\Delta_{D^{\nu}}$ is the different, 
and the fact that $(K_S + \Delta)|_{D^{\nu}}\equiv 0$.\\

\emph{Step 2}.
We prove (iv), that is, assume that $\dim Z\ge 1$. Since $D$ is contracted to a point on $Z$, $D^2\le 0$: 
this can be seen by taking an effective Cartier divisor $L$ on $Z$ passing through $f(D)$ and considering $(f^*L)\cdot D=0$. 
Letting 
$$
\Gamma=\Delta+(1-\mu_D\Delta)D,
$$ 
and applying adjunction 
$$
K_{D^{\nu}}+\Gamma_{D^{\nu}}:=(K_S + \Gamma)|_{D^{\nu}},
$$ 
we have $\Gamma_{D^{\nu}}\ge 0$ and 
$$
\deg (K_{D^{\nu}}+\Gamma_{D^{\nu}})=(K_S+\Delta)\cdot D+(1-\mu_D\Delta)D^2\le 0.
$$ 
This implies that $D^\nu$ is either a rational curve 
or an elliptic curve. Note that $(S,\Gamma)$ may not be lc but the adjunction formula still makes sense.\\

\emph{Step 3}.
From now on we assume that $\dim Z=0$. By assumption $\Delta\neq 0$, so $S$ is uniruled. 
 Running an MMP on $K_{S}$ and replacing $S$ with the resulting model we can assume that we have a Mori fibre 
space $S\to V$. Note that if the MMP contracts $D$, then it is a rational curve and we are done, so we 
can assume that $D$ is not contracted.  
Now $S$ is either $\mathbb{P}^2$ or a {minimal} ruled surface over a smooth curve. 
If $S=\mathbb{P}^2$, then the degree of $D$ is bounded from above as $K_{S}+\Delta\equiv 0$ and as 
the coefficient of $D$ in $\Delta$ is $\ge u$, hence 
$D$ belongs to a bounded family which implies its gonality is bounded and that in fact 
the genus of its normalization $D^{\nu}$ is bounded as well. 
So we can assume that $S$ is a {minimal} ruled surface over $V$.\\ 

\emph{Step 4}.
Denote $S\to V$ by $h$. Consider the canonical bundle formula 
$$
K_S+\Delta\sim_\R h^*(K_V+\Delta_V+M_V)
$$
where $\Delta_V,M_V$ are the discriminant and moduli divisors; $\Delta_V$ is effective and 
$M_V$ is pseudo-effective (cf. Theorem 3.6 of \cite{B19}). Since 
$$
\deg (K_V+\Delta_V+M_V)=0,
$$ 
$V$ is either $\PP^1$ or an elliptic curve. 

If $D$ is a fibre of $S\to V$, then it is rational. So we assume $D$ is horizontal over $V$. Then the intersection 
of $D$ with a fibre of $h$ 
is bounded  as $K_{S}+\Delta\equiv 0$ and as
the coefficient of $D$ in $\Delta$ is $\ge u$.
Thus the induced map $D\to V$ has bounded degree depending only on $u$, so the gonality of 
$D$ is bounded depending only on $u$. We have then proved (i).\\
 
\emph{Step 5}.
It remains to prove (ii), that is, assuming $u>\frac{1}{2}$ when $S$ is a minimal ruled surface over $V$ 
as in Step 4. In this case the degree of $D\to V$ is at most $3$. 
We can assume this degree is $2$ or $3$ otherwise $D\to V$ is birational in which case $g(D^{\nu}) \leq 1$. 

Until the end of this step we assume that $V$ is an elliptic curve. 
We claim that $D^2$ is bounded from above depending only on $u$. 
If $D^2\le 0$, the claim is obvious, so assume $D^2>0$ which in particular means $D$ is a nef and big divisor. 
Pick $\epsilon>0$ so that $u>\frac{1}{2}+\epsilon$. 

Suppose ${\rm vol}(\epsilon D)=\epsilon^2 D^2>16$. Then we can find $0\le L\sim_\R \epsilon D$ so that letting  
$$
\Theta:=\Delta-\epsilon D+L\sim_\R \Delta,
$$ 
the pair $(S,\Theta)$ is lc but not klt (cf. Lemma 7.1 of \cite{HMX14}). Let $T$ be a non-klt centre of $(S,\Theta)$.
Assume $T$ is horizontal over $V$ in which case $T$ is a curve. Then $T=D$  otherwise if $f$ is a fibre of $h$, then 
$$
\Theta\cdot f\ge ((u-\epsilon)D+T)\cdot f\ge 2(u-\epsilon)+1>2,
$$
contradicting $K_S+\Theta\equiv 0$. Applying adjunction by restricting $K_S+\Theta$ to $D^\nu$ as in Step 1, 
we see that the genus of $D^\nu$ is $\le 1$. Now assume $T$ is vertical over $V$. Then in the 
canonical bundle forumula 
$$
K_S+\Theta\sim_\R h^*(K_V+\Theta_V+N_V),
$$
we get $\Theta_V\ge h(T)$ (by definition of the discriminant divisor) and $\deg N_V\ge 0$, 
contradicting the assumption that $V$ is an elliptic curve and that 
$$
\deg K_V=\deg(K_V+\Theta_V+N_V) =0.
$$
 Therefore, $\epsilon^2 D^2\le 16$, so $D^2$ is bounded from above as claimed. 

Now letting 
$$
\Gamma=\Delta+(1-\mu_D\Delta)D
$$
 and applying adjunction
$$
K_{D^{\nu}}+\Gamma_{D^{\nu}}:=(K_S + \Gamma)|_{D^{\nu}},
$$ 
we have $\Gamma_{D^{\nu}}\ge 0$ and 
$$
\deg (K_{D^{\nu}}+\Gamma_{D^{\nu}})=(K_S+\Delta)\cdot D+(1-\mu_D\Delta)D^2=(1-\mu_D\Delta)D^2,
$$ 
so $\deg K_{D^{\nu}}$ is bounded from above, so the genus of $D^\nu$ is bounded as required.\\   
 
\emph{Step 6}. 
From now on we assume $V=\PP^1$. So $S$ is a Hirzebruch surface $\mathbb{F}_n$ for some $n\geq 0,\ n\neq 1$.  
Hence we have $D\sim as + bf$ where $a\geq 2,\ b\geq 0$, $f$ is a fibre of $h$ as above, and $s$ is the 
section of $h$ with $s^2=-n$. 
Since $K_{S} + \Delta\equiv 0$, we have 
$$
2-ua=-(K_{S}+uD)\cdot f\ge 0,
$$ 
and since $u>\frac{1}{2}$, we see that $a\le 3$ (actually $a$ is just the degree of $D\to V$, so it is 
$2$ or $3$). 
Also we can assume $D\neq s$ otherwise $D$ would be a rational curve. Then 
$$
D \cdot s = -an + b \geq 0, 
$$ 
hence $b\geq an$. Moreover,
\[
-K_{S} - u D=2s + (2+n)f-uD\sim_{\mathbb{Q}} (2-ua)s + (2+n-ub)f
\] 
is $\mathbb{R}$-linearly equivalent to the effective divisor $\Delta-uD$. Since the cone of effective curves on $\mathbb{F}_n$ is generated by $s$ and $f$, this implies that $2-ua\ge 0$ and $2+n-ub\ge 0$. Thus 
$2+n\geq ub\geq u an$. We obtain $2 \geq n (ua - 1)$. Recall that $a \geq 2$, $u>\frac{1}{2}$ and  $ua- 1\ge 2u-1>0$. 
Thus 
$$
n \leq 2/(ua-1)\le 2/(2u-1)
$$ 
which bounds $n$ in terms of $u$. Since $ub \leq 2+n$ it follows that $b$ is bounded as well. Hence $D$ belongs to a finite number of algebraic families. In  other words, the genus of its normalization is bounded. 

\end{proof}

\begin{remark}\label{rem-dim-2} 
{ In Step 5 of the proof of Theorem \ref{t-general-lcy-dim-2} we showed that 
$g(D^\nu)$ is bounded from above depending only on $u$ when $u>\frac{1}{2}$ and $S$ is a minimal ruled surface 
over an elliptic curve $V$. We will argue that in fact in this case we have $g(D^\nu)\leq 1$. First note that $\mathrm{NS}(S)$ is generated by the minimal section $s$ and the fiber $f$ such that $-e=s^2\leq 1$, $s\cdot f = 1$, $f^2=0$. We have $K_S\equiv -2s - ef$. Write $D\equiv as + bf$ where we can assume $a\geq 2$ otherwise $a=0$ in which case $D$ is a rational curve (as $D$ would be a fibre) or $a=1$ in which case 
$D$ is an elliptic curve (as $D\to V$ would be birational).

Since $K_{S} + \Delta\equiv 0$, we have 
$$
2-ua=-(K_{S}+uD)\cdot f  \ge 0,
$$ 
and since $u>\frac{1}{2}$, we see that $a\le 3$. Actually $a$ is just the degree of $D\to V$, so it is 
$2$ or $3$.  Then 
$$
D \cdot s = -ae + b \geq 0, 
$$ 
hence $b\geq ae$. Moreover,
\[
-K_{S} - u D\equiv 2s + ef-uD\equiv (2-ua)s + (e-ub)f
\] 
is numerically equivalent to the effective divisor $\Delta-uD$. We will consider the two cases $e\geq 0$ 
and $e=-1$ separately. 

In the first case, the cone of effective curves on $S$ is generated by $s$ and $f$ (cf. 
\cite{H77}, {Chapter V, Proposition 2.20}). This implies that $2-ua\ge 0$ and $e-ub\ge 0$. Thus 
$e\geq ub\geq uae$. If $e\neq 0$, we obtain $ua\leq 1$. Recall that $a \geq 2$, $u>\frac{1}{2}$ and thus $ua>1$. This is a contradiction. So $e=0$ and hence $b=0$. We conclude that $K_S = -2s$, and by adjunction $\deg K_D = (-2+a)a s^2 = 0$, so $g(D^\nu)\leq 1$. 

In the second case, the cone of curves on $S$ is generated by $2s-f$ and $f$ (cf. 
\cite{H77}, {Chapter V, Proposition 2.21}). Note that $2s-f =-K_S$ and $-K_S \equiv uD + (\Delta -uD)$. 
This in particular implies that $D$ is numerically proportional to $-K_S$ as $-K_S$ generates an extremal ray of the cone of effective curves. We can then write $D\equiv d(2s-f)$ where $d\in \mathbb{Q}^{>0}$ and $d\leq 1/u$. By adjunction 
\[
\deg K_D = (K_S + D)\cdot D = (1+d) d (2s-f)^2 = 0,
\]
hence $g(D^\nu)\leq 1$.}

\end{remark}

{\begin{exam}\label{exam-u-greater-one-half} 
We show that the condition $u>\frac{1}{2}$ is important for the second assertion of Theorem \ref{t-general-lcy-dim-2}. Consider the klt surface pair $(S, \frac{1}{2}C)$ where $S = \mathbb{P}(1,1,n)$ for $n\geq 3$ and $C$ is a smooth hyperelliptic curve given by the  equation $z^2 = f_{2n}(x, y)$ where $x,y,z$ are homogeneous coordinates on $S$ with weights $1,1,n$, respectively, and the homogeneous polynomial $f_{2n}$ is general and has degree $2n$. One has $C\sim 2n h$ where $h$ is the positive generator of the class group $\mathrm{Cl}(S)$. One easily computes that $g(C)=n-1$, and 
$$
-(K_S+\frac{1}{2}C) \sim (n+2)h - nh = 2h
$$ 
is ample. In particular, the genus $g(C)$ is not bounded, but the gonality of $C$ is equal to $2$. 
Take $\Delta\ge \frac{1}{2}C$ so that $(S,\Delta)$ is lc and $K_S+\Delta\sim_\Q 0$ and take $Z$ to be a point. 
Then $(S,\Delta)\to Z$ is a log Calabi-Yau fibration, $C$ is a component of $\Delta$ with coefficient $\ge \frac{1}{2}$, 
and $C$ is obviously contracted to a point on $Z$. So we see that 
Theorem \ref{t-general-lcy-dim-2} does not hold if we remove the condition $u>\frac{1}{2}$.

\end{exam}

{In the next proposition, we give effective bounds for Theorem \ref{t-general-lcy-dim-2}.} 
We only need to consider cases (i) and (ii) because in cases (iii) and (iv) genus of $D^\nu$ is at most $1$.

\begin{proposition}\label{p-dim2-effective-bnd}
Fix a real number $u>0$. Assume that $(S, \Delta)$ is an lc log Calabi-Yau surface pair, and that the coefficient of $D$ in $\Delta$ is $\geq u$. Then 
\[
\mathrm{gon}(D) \leq 3/u - 1,
\]
and this bound is sharp. If, moreover, $u>1/2$, then
\[
g(D^\nu)\leq \frac{1}{u-1/2} - 1,
\] 
and this bound is sharp. In the latter case, if $g(D^\nu)>1$ then $S$ is rational.
\end{proposition}
\begin{proof}
{
We use notation and ideas similar to the proof of \ref{t-general-lcy-dim-2}. We may assume that we have a Mori fibre 
space structure $S\to V$ and that $S$ is smooth. First we treat the gonality. Assume $V$ is a point in which case $S=\mathbb{P}^2$. 
Since $K_S+uD$ is anti-nef, $-3+u\deg D\le 0$ where $\deg D$ stands for degree of $D$ as a plane curve. 
It follows that either $\mathrm{gon}(D)=1$ or 
\begin{equation}
\label{bound-gon-1}
\mathrm{gon}(D)\leq 3/u - 1,
\end{equation}
and this bound is sharp, e.g. equality holds when $u=1$ and $D$ is an elliptic curve, where we 
make use of the fact $\mathrm{gon}(D)\le \deg D-1$.

Assume that $V$ is a curve. Let $f$ be a general fiber of $S\to V$. It was proven that either $V$ is a smooth rational curve, or a smooth elliptic curve. In the first case, we have 
\[
0 = ( K_S + \Delta )\cdot f = ( K_S + u D + (\Delta-uD) + f )\cdot f \geq \deg K_f + u D \cdot f  
\]
hence $D\cdot f \leq 2/u$ and thus 
\begin{equation}
\label{bound-gon-2}
\mathrm{gon}(D)\leq 2/u.
\end{equation}
This bound is sharp, e.g. when $S=\mathbb{P}^1\times \mathbb{P}^1$, $u=1$, and $D$ is a general curve in 
$|-K_S|$ in which case $D$ is an elliptic curve.

Now assume that $V$ is an elliptic curve. Then since $K_S+uD$ is anti-nef over $V$, intersecting with a general fibre 
we see that $-2+ua\leq 0$ where $a$ is the degree of the map $D\to V$. So again $\mathrm{gon}(D)\leq 2/u$. We compare the two obtained bounds for gonality \eqref{bound-gon-1} and \eqref{bound-gon-2}. Since $u\leq1$, we have $3/u - 1\geq2/u$, so the overall bound for gonality is $\mathrm{gon}(D)\leq 3/u - 1$.

\

We  now give a bound for the genus of $D^\nu$ under the assumption $u>1/2$. If $g(D^\nu)\le 1$, then 
$$
g(D^\nu)\le 1 \leq \frac{1}{u-1/2} - 1,
$$
so we can assume that $g(D^\nu)\ge 2$. 
 Put $\alpha = u-1/2$. Assume that $\dim V=1$ where $V$ is as above. Consider the case $V=\mathbb{P}^1$, hence $S=\mathbb{F}_n$ for some $n\geq 0$. 
Write $D\sim as + bf$ where $s,f$ are as in the proof of Theorem \ref{t-general-lcy-dim-2} ($s$ is a 
section and $f$ is a fibre of $S\to V$). Then 
\begin{equation}
\begin{split}
2p_a(D) - 2 =& (K_{\mathbb{F}_n} + D)\cdot D \\
=& ( ( a - 2 ) s + ( b - 2 - n ) f ) ( as + bf) \\
=& -a^2 n + a n + 2 a b - 2 a - 2 b,
\end{split}
\end{equation}
so
\[
p_a(D) = \frac{1}{2}an(1-a) + (a-1)(b-1).
\]
From the proof of \ref{t-general-lcy-dim-2} it follows that  
\[
u b - 2 \leq n\leq \frac{2}{2u-1} = \frac{1}{\alpha}, \ \ \ 2\le a\leq 3, \ \ \ b\leq \frac{2+n}{u} 
\le \frac{2}{u-1/2}=\frac{2}{\alpha}.
\]
Consider the two cases: $a=2$ and $a=3$ ($a=1$ implies that $D$ is rational). If $a=2$ we obtain 
\begin{equation}
\begin{split}
\label{bound-genus}
p_a(D)=& - n + b - 1
\leq -ub + 2 + b - 1 \\ 
=& b ( 1 - ( 1/2 + \alpha ) ) + 1 
\leq 2/\alpha ( 1/2 - \alpha ) + 1 
= 1/\alpha - 1  .
\end{split}
\end{equation}
If $a=3$ we get 
\begin{equation}
\begin{split}
p_a(D) =& -3 n + 2 b - 2 \leq 3 (-ub + 2) + 2b - 2 \\
=& b ( 2 - 3 ( 1/2 + \alpha ) ) + 4 \leq 2/\alpha ( 1/2 - 3\alpha ) + 4 = 1/\alpha - 2.
\end{split}
\end{equation}

We show that the bound \eqref{bound-genus} is sharp. Indeed, as in Example \ref{exam-u-greater-one-half} take $S = \mathbb{P}(1,1,n)$, $D\sim 2nh$, $\alpha = 1/n$, $u = 1/2+\alpha$. Then $(S, (1/2+\alpha)D = uD)$ is 
an lc log Calabi-Yau pair and $g(D) = n - 1 = 1/\alpha - 1$. 

Now consider the case when $V$ is an elliptic curve. 
From Remark \ref{rem-dim-2} it follows that $g(D^\nu)\leq 1$. 

Finally assume that $\dim V=0$. Then $-3+ud\le 0$ where $d=\deg D$. 
Since $u>1/2$, we have $d\leq 5$. Then 
\[
p_a(D)= \frac{1}{2}(d-1)(d-2)\in\{0,1,3,6\} 
\]
which gives $g(D^\nu)\leq p_a(D)\le 6$. To verify the bound in the proposition we only need to check the cases 
$d=4,5$ because in the other cases we have $g(D^\nu)\le 1$. 
When $d=4$, we have $u\le 3/4$, so $u-1/2\le 1/4$, hence we can see that 
$$
g(D^\nu)\le 3\le \frac{1}{u-1/2} - 1.
$$ 
On the other hand, when $d=5$, we have $u\le 3/5$, so $u-1/2\le 1/10$, hence we can see that 
$$
g(D^\nu)\le 6\le 9\le \frac{1}{u-1/2} - 1.
$$ 
}
\end{proof}


\section{Preparations}

In this section we make some preparations towards proving the main results in dimension $3$. 

\subsection{Dlt models}

{By a \emph{dlt model} of an lc pair $(X, B)$ we mean a $\mathbb{Q}$-factorial dlt pair $(Y,B_Y)$ together with a birational contraction $Y \to X$ such that $K_Y + B_Y$ is the pullback of $K_X + B$.}

\begin{lem}\label{l-Q-fact-dlt-model}
Let $(X,B)$ be a $\Q$-factorial lc pair and $(Y,B_Y)$ be a $\mathbb Q$-factorial dlt model of $(X,B)$. 
Then the exceptional locus of $Y\to X$ is contained in $\rddown{B_Y}$.
\end{lem}
\begin{proof}
Since $X$ is $\Q$-factorial,  
the exceptional locus of $Y\to X$ coincides with the exceptional part of $\rddown{B_Y}$: indeed, 
$X$ being $\Q$-factorial, there is an effective exceptional divisor which is anti-ample 
over $X$, hence the exceptional locus of $Y\to X$ is equal to the reduced exceptional divisor;  
on the other hand, each component of the reduced exceptional divisor is a component of 
$\rddown{B_Y}$ by definition of dlt models. 

\end{proof}

\begin{lem}\label{l-Q-fact-dlt-model-2}
Let $f\colon (X,B)\to Z$ be a Fano type log Calabi-Yau fibration where $\dim Z\ge 1$. 
Assume $D$ is a prime divisor over $X$ such that $a(D,X,B)< 1$ and the image of $D$ on $Z$ is a closed point. 
Then there is a Fano type log Calabi-Yau 
fibration $(Y,B_Y)\to Z$ and a birational map $X\bir Y$ over $Z$ such that 
\begin{itemize}
\item $(Y,B_Y)$ is $\mathbb Q$-factorial dlt,
\item $a(D,Y,B_Y)\le a(D,X,B)$, and 

\item the centre of $D$ on $Y$ is contained in $\rddown{B_Y}$.
\end{itemize}
\end{lem}
\begin{proof}
Note that, as we will see during the proof, $B_Y$ is not necessarily the birational transform of $B$.
Replacing $(X,B)$ with a $\Q$-factorial dlt model, we can assume $X$ is $\Q$-factorial. This 
preserves the Fano type property by a relative version of \cite[2.14(7)]{B19}.
Extracting $D$ we can assume it is a divisor on $X$ with positive coefficient in $B$. 
Since $D$ is contracted to a point $z$ on $Z$ and $\dim Z\ge 1$, 
$D$ does not intersect the general fibres of $X\to Z$. In particular, $-D$ is pseudo-effective over $Z$ 
as we can find a Cartier divisor $L\ge 0$ on $Z$ such that $D+f^*L\ge 0$.
Since $X$ is of Fano type over $Z$, we can run an MMP on $-D$ over $Z$ which ends with a minimal model 
for $-D$ on which the pushdown of $-D$ is semi-ample over $Z$. The MMP does not contract $D$ 
by the negativity lemma. Replacing $X$ with the minimal model 
we can assume that $-D$ is semi-ample over $Z$. 

Let $s$ be the lc threshold of $D$ with respect to $(X,B)$. 
We can find a general $0\le A\sim_\Q -D/Z$, so that $(X,B+sD+sA)$ is lc. Replacing $B$ with $B+sD+sA$ 
we can assume $D$ contains an lc centre of $(X,B)$. Thus there is a prime divisor $S$ over $X$ 
such that $a(S,X,B)=0$ and the centre of $S$ on $X$ is contained in $D$.  

If $S$ is not exceptional over $X$ (i.e. if $S=D$), let $(V,B_V)=(X,B)$.
But if $S$ is exceptional over $X$, let $W\to X$ be the birational contraction 
from $\Q$-factorial $W$ which contracts $S$ but no other divisors: this exists by \cite[Corollary 1.4.3]{BCHM10}; 
note that $W$ is of Fano type over $Z$ by a relative version of \cite[2.14(7)]{B19}; then 
let $V$ be the minimal model of $-S$ over $Z$ (abusing notation we denote the birational 
transform of $S$ on $V$ again by $S$); let $K_W+B_W$ be the pullback of $K_X+B$ and 
$K_V+B_V$ the pushdown of $K_W+B_W$. By construction, in any case, $S$ is a component of $\rddown{B_V}$, 
$-S$ is nef over $Z$, and $V$ is $\Q$-factorial.

Let $T$ be the centre of $D$ on $V$. We will show that $T$ is contained in $S$.
As $V$ is of Fano type over $Z$, $-S$ is semi-ample over $Z$ defining a 
contraction $g\colon V\to M/Z$ where $M\to Z$ is birational because $S$ is 
contracted to the point $z$ in $Z$ so $S$ does not intersect the general fibres of $V\to Z$. 
Since $S\sim_\Q 0/M$ and $S$ is vertical over $M$, we have $S=g^*E$ for some effective $\Q$-divisor $E$ on $M$. 
Since $S$ is irreducible, $E$ is also irreducible. Moreover, 
$E$ is anti-ample over $Z$, hence its support contains the exceptional locus of $M\to Z$.  
Thus since $E$ is contracted to $z$, the support of $E$ is equal to the inverse image of $z$ under the map $M\to Z$. 
By assumption, $T$ is mapped to $z$ by $V\to Z$. 
Thus $T$ is mapped into $E$ by $V\to M$ which in turn implies that $T$ is contained in $S=g^*E$. 

Let $(Y,B_Y)$ be a $\mathbb Q$-factorial dlt model of $(V,B_V)$. 
By construction, 
$$
a(D,Y,B_Y)=a(D,V,B_V)=a(D,W,B_W)=a(D,X,B)<1.
$$
Recall that we increased $B$ in the above arguments, so the current $a(D,X,B)$ 
is less than or equal to the original log discrepancy $a(D,X,B)$ we started with.
It remains to show that the centre of $D$ on $Y$ is contained in $\rddown{B_Y}$.

If $Y\to V$ is an isomorphism over the generic point of $T$, then the centre of $D$ on $Y$ is contained in 
$\rddown{B_Y}$ since $T$ is contained in $S$. But if $Y\to V$ is not an isomorphism over the generic point of $T$,
then the centre of $D$ on $Y$ is contained in the exceptional locus of $Y\to V$, hence it is contained in 
$\rddown{B_Y}$ as $V$ is $\Q$-factorial, by Lemma \ref{l-Q-fact-dlt-model}. 

\end{proof}

\subsection{Divisors over dlt pairs}

The next lemma is crucial for the proofs of the main results.

\begin{lem}
\label{plt_bounded}
Let $s\in [0,1)$ be a real number. 
Assume that  
\begin{itemize}
\item $(Y, B)$ is a $\mathbb{Q}$-factorial {dlt} pair of dimension $3$, 

\item $E$ is a prime divisor over $Y$ with centre $T$ which is a curve, 

\item $T$ is contained in $\left \lfloor{B} \right \rfloor$, 

\item the log discrepancy $a(E, Y, B) \le s$.
\end{itemize}
Then there exists a modification $\phi\colon Y' \to Y$ such that 
\begin{itemize}
\item
$Y'$ is $\mathbb{Q}$-factorial and $\rho(Y'/Y)=1$, 
\item
the only $\phi$-exceptional divisor is $E$,
\item
the number of components in the general fibre of $E\to T$ is bounded depending only on $s$, 
\item
if $s < 1/2$, then the general fibre of $E\to T$ is irreducible. 
\end{itemize}
\end{lem}
\begin{proof}
Increasing $s$ slightly we can assume $a(E, Y, B) < s$. 
Since $(Y, B)$ is $\mathbb{Q}$-factorial {dlt}, one has that $(Y,0)$ is klt. 
Moreover, by assumption 
$T$ is contained in some component $B_1$ of $\left \lfloor{B} \right \rfloor$, so perturbing the coefficients of 
$B$ in components other than $B_1$ we can assume $(X,B)$ is plt. 
 Existence of a modification $\phi\colon Y' \to Y$,  with the first two properties follows from 
\cite[Corollary 1.4.3]{BCHM10}. 
In particular, the $\mathbb Q$-factoriality assumption ensures that the exceptional locus of $\phi$ 
is exactly $E$, i.e. no curves outside $E$ are contracted.

Let us prove that the third property holds. Let $x$ be a general point on $T$. 
The fibre of $E\to T$ over $x$ which is the same as the fibre of $Y'\to Y$ over $x$, say $E_x$, is rationally chain connected  by \cite{HM07} because $Y'$ is of Fano type over $Y$ as $(Y, B)$ is plt. Let $E_x = E_1 + \ldots + E_M$ be a decomposition of $E_x$ into the sum of irreducible components. We will show that $M$ is bounded. Since $\rho(Y'/Y)=1$ 
and since $B_1$ contains $T$, 
the strict transform $B'_1$ of $B_1$ is ample over $Y$. Hence $B'_1$ intersects each $E_i$. 

Both $B_1$ and $B_1'$ are normal by the plt property. Thus the morphism $B'_1 \to B_1$ is a birational contraction (by Zariski's main theorem). In particular, $B'_1 \to B_1$ is an isomorphism over the generic point of $T$. Since $B_1'$ intersects every fibre of $E\to T$, $E\cap B'_1$ maps onto $T$. 
Thus there is a unique component $D$ of $E\cap B'_1$ such that the map $D \to T$ is birational. 
In particular, since $x$ is general,  $E_x\cap D$ is a single point, say $x'$. 
Moreover, $E_x\cap D=E_x\cap B_1'$. 
Therefore, since $B'_1$ intersects each $E_i$, we see that the set-theoretic intersection 
$E_i\cap B_1'$ consists of exactly $x'$ for each $E_i$, so
 all the components $E_i$ of the fibre $E_x$ pass through $x'$. 

Consider a general hyperplane section $H$ on $Y$ and its preimage $H'$ on $Y'$. Clearly, $H,H'$ are normal. 
Apply adjunction to get 
\[
K_{H} + B_H = \left( K_Y + B+H \right)|_H
\]
and
\[
K_{H'} + B_{H'} = \left( K_{Y'} + B'+H' \right)|_{H'}
\]
where $K_{Y'}+B'$ is the pullback of $K_Y+B$. 
The pairs $(H, {B_H})$ and $(H', {B_{H'}})$ are dlt and  
\[
K_{H'} + B_{H'} = \phi|_{H'}^*\left(K_H + B_H\right).
\]
 We may assume that $x\in H$ and hence $x' \in H'$. Then over a neighbourhood of $x$ we have 
 $$
 E\cap H' = E_x=E_1 + \ldots + E_M.
 $$ 
Thus each $E_i$ is a component of $B_{H'}$ with coefficient $>1-s>0$: indeed, 
the coefficient of $E_i$ 
is at least $1-\frac{1}{n}+\frac{e}{n}$ where $n$ is the Cartier index of $K_{Y'}+H'$ along $E_i$ 
and $e$ is the coefficient of $E$ in $B'+H'$ (see 3.9 and 3.10 of \cite{Sh92}). 
But 
$$
1-\frac{1}{n}+\frac{e}{n}\ge e=1 - a(E,Y,B)>1-s>0. 
$$
On the other hand, $B_1'\cap H'$ gives a component of $B_{H'}$ with coefficient $1$  
and distinct from each $E_i$. This component passes through $x'$ because 
$$
x'\in B_1'\cap E_x\subset B_1'\cap H'.
$$ 

We have shown that $B_{H'}$ has a component $R'$ with coefficient $1$ and $M$ components 
with coefficient $>1-s$ all passing through $x'$. But then applying adjunction 
$$
K_{R'}+B_{R'}=(K_{H'}+B_{H'})|_{R'}
$$
and the fact that $(R',B_{R'})$ is lc shows that $M(1-s)<1$ which implies that $M$ is 
bounded depending only on $s$. If $s<\frac{1}{2}$, then $1-s>\frac{1}{2}$, so $M=1$.

\end{proof}

\subsection{An MMP over a pair}
\label{construction_mmp}

Fix a positive real number $t$. 
Let $(Y,B_Y)$ be a $\Q$-factorial dlt pair. 
Assume $D$ is an exceptional prime divisor over $Y$ such that $a(D,Y,B_Y)\le 1-t$. We will find
\begin{itemize}
\item a log smooth 
pair $(V,\Gamma)$ over $Y$ where every prime exceptional divisor of $V\to Y$ is a component of $\rddown{\Gamma}$, 
and the centre of $D$ on $V$ is a divisor (so it is the birational transform of $D$) hence this centre is a component of $\rddown{\Gamma}$, 
\item an MMP 
$$
V=V_1\bir V_2\bir \cdots \bir V_N=Y
$$ 
on $K_V+\Gamma$  over $Y$,
\item $0<u<t$ depending only on $t$ such that for each $k$,
$$
a(D,V_k,\Gamma_k)\le a(D,Y,(1-\epsilon)B_Y)< 1-u
$$  
where $\Gamma_k$ is the pushdown of $\Gamma$ and $0<\epsilon \ll 1$, and 
\item  in addition, if $t>1/2$, we can assume $u>1/2$. 
\end{itemize}

Let $\psi\colon V \to Y$ be a log resolution of $(Y,B_Y)$ on which $D$ is a divisor. 
Define 
\[
\Gamma=\left(1-\epsilon\right)B_{Y}^\sim+\sum E_i.
\]
for some $0<\epsilon\ll1$ where $B_{Y}^\sim$ is the birational transform of $B_Y$ and 
$E_i$ are the exceptional prime divisors of $\psi$. 

The pair $(Y, (1-\epsilon)B_Y)$ is klt. 
Thus we can write 
\[
K_V + (1-\epsilon)B_{Y}^\sim = \psi^* (K_{Y} + (1-\epsilon)B_Y) + \sum a_i E_i
\]
where $a_i>-1$. Then 
\begin{equation}
\label{ineq-mmp-discr}
\begin{split}
K_V + \Gamma & = K_V + \left(1-\epsilon\right)B_{Y}^\sim+\sum E_i  \\
& = K_V + \left(1-\epsilon\right)B_{Y}^\sim - \sum a_i E_i+\sum (1+a_i)E_i \\
& =\psi^* \left(K_{Y} + \left(1-\epsilon\right)B_Y\right)+\sum (1+a_i)E_i.
\end{split}
\end{equation}

Run the $(K_V + \Gamma)$-MMP over $Y$. This terminates with $Y$, by Theorem 1.8 of \cite{B12}, 
since $Y$ is $\Q$-factorial and since 
 $\sum (1+a_i)E_i$ is effective and its support is equal to the reduced exceptional divisor of $\psi$. Let 
 the steps of the MMP be $V_k\bir V_{k+1}$ where $V_1=V$ and $V_N=Y$.   
We can ensure that there is $0<u<t$ depending only on $t$ such that for each $k$,
$$
a(D,V_k,\Gamma_k)\le a(D,Y,(1-\epsilon)B_Y)< 1-u
$$  
where $\Gamma_k$ is the pushdown of $\Gamma$: the first inequality follows from the definition of $\Gamma$ and the formula {\ref{ineq-mmp-discr}}.
The second inequality follows from the assumption that $\epsilon$ is sufficiently small (note that $\epsilon$ 
depends on the pair $(Y,B_Y)$). 
 In addition, if $t>1/2$, we can assume $u>1/2$.

\section{Proof of main results}

In this section we prove our main results in dimension 3.

\begin{proof}(of Theorem \ref{the_theorem-general-lcy})
\emph{Step 1}.
First we prove (iii), that is, assume $t=1$. 
Let $S$ be the normalization of $D$. By adjunction we can write   
$K_{S}+B_{S}= (K_X+B)|_{S}$ where $(S,B_S)$ is an lc pair. 
Since $(X,B)\to Z$ is a log Calabi-Yau fibration and since $D$ is contracted to a point on $Z$, 
$K_S+B_S\sim_\R 0$. Thus $(S,B_S)$ is log Calabi-Yau of dimension 2. 
Since $X$ is of Fano type over $Z$, its fibres are rationally chain connected \cite{HM07}, hence 
$D$ is uniruled which in turn means $S$ is uniruled (note that since $\dim Z\ge 1$, $D$ is a component of a fibre of $X\to Z$). 
This implies that $(S,B_S)$ is not canonical, that is, it is not $1$-lc: indeed take the minimal resolution 
$\pi\colon S'\to S$ and let $K_{S'}+B_{S'}=\pi^*(K_S+B_S)$; if $(S,B_S)$ is canonical, then $B_{S'}=0$,  
hence $K_{S'}\sim_\Q 0$ contradicting the fact that $S$ is uniruled.  
Now applying Lemma \ref{lc_log_cy_bounded} we deduce that $S$ (hence $D$) is birational to $\PP^1\times C$ 
where $C$ is a smooth projective of genus $\le 1$. 
It remains to prove (i) and (ii).\\ 

\emph{Step 2}.
By assumption, 
$$
a(D,X,B)=1-\mu_DB\le  1-t<1.
$$
 Applying Lemma \ref{l-Q-fact-dlt-model-2}, there is a Fano type log Calabi-Yau 
fibration $(Y,B_Y)\to Z$ and a birational map $X\bir Y$ such that 
\begin{itemize}
\item  $(Y,B_Y)$ is $\mathbb Q$-factorial dlt, 
\item $a(D,Y,B_Y)\le a(D,X,B)$, and  
\item the centre of $D$ 
on $Y$, say $T$, is contained in $\rddown{B_Y}$.
\end{itemize}
In particular,  $a(D,Y,B_Y)\le 1-t$.

Assume $\dim T=2$. Then $T$ is the birational transform of $D$ and $T$ is a component of $\rddown{B_Y}$.
Applying (iii) to $(Y,B_Y)\to Z$, we deduce that $D$ is birational to $\PP^1\times C$ 
where $C$ is a smooth projective curve of genus $\le 1$. So we are done in this case.\\ 

\emph{Step 3}.
Assume $\dim T=1$. Applying Lemma \ref{plt_bounded} to $(Y,B_Y),\ E:=D$, and $s:=1-t$, 
there exists a modification $\phi\colon Y' \to Y$ such that 
\begin{itemize}
\item
$Y'$ is $\mathbb{Q}$-factorial and $\rho(Y'/Y)=1$, 
\item
the only $\phi$-exceptional divisor is $E$ which is the birational transform of $D$,
\item
the number of components in the general fibre of $E\to T$ is bounded depending only on $1-t$, 
\item
if $1-t < 1/2$, that is, if $t>\frac{1}{2}$, then the general fibre of $E\to T$ is irreducible. 
\end{itemize}

Consider the normalization $E^\nu$ of $E$ and the Stein factorisation $E^\nu \to C\to T$. Then the general 
fibres of $E^\nu \to C$ are $\PP^1$ because the fibres of $E\to T$ are rationally chain connected. 
So $E^\nu$ is birational to $\PP^1\times C$ which means $D$ is also 
birational to  $\PP^1\times C$. Moreover, the degree of $C\to T$ is equal to 
the number of components of the general fibres of 
$E^\nu \to T$ which is in turn equal to the number of components of the general fibres of $E\to T$ 
because $E^\nu \to E$ is an isomorphism over the generic point of any component of a general fibre of $E\to T$. 
Thus the degree of $C\to T$ is bounded depending only on $t$, and if $t>\frac{1}{2}$, then 
this degree is $1$, that is, in this case $C\to T$ is birational. Therefore, to finish the proof 
in the case $\dim T=1$ we only need to show that the gonality of $T$ is bounded depending only on $t$, and in case $t>\frac{1}{2}$ the 
genus of the normalization of $T$ is bounded.

By construction, $T$ is contained in some component of $\rddown{B_Y}$, say $S$. 
Let $S\to R\to Z$ be the Stein factorisation of $S\to Z$. By adjunction, we can write   
$K_{S}+B_{S}= (K_Y+B_Y)|_{S}$ where $(S,B_S)$ is an lc pair which is log Calabi-Yau over $R$. 
We claim that the coefficient of $T$ in $B_S$ is $\ge t$. Consider the birational transform $S'$ on $Y'$ 
and then consider the adjunction $K_{S'}+B_{S'}= (K_{Y'}+B_{Y'})|_{S}$ where $K_{Y'}+B_{Y'}$ is 
the pullback of $K_Y+B_Y$ (note that $S'$ is normal as it is easy to check that $(Y',S')$ is plt). 
Then $B_S$ is the pushdown of $B_{S'}$. Since $S'$ is ample over $Y$, it intersects every fibres of $E\to T$, 
so there is a component $T'$ of $S'\cap E$ which maps onto $T$. 
Since the coefficient of $E$ in $B_{Y'}$ is more than or equal to the coefficient of $D$ in $B$ 
which is $\ge t$, we see that  the coefficient of $T'$ in $B_{S'}$ is 
at least $1-\frac{1}{n}+\frac{t}{n}$ where $n$ is the Cartier index of $K_{Y'}+S'$ along $T'$ 
(see 3.9 and 3.10 of \cite{Sh92}). But $1-\frac{1}{n}+\frac{t}{n}\ge t$, 
hence the coefficient of $T'$ in $B_{S'}$ is $\ge t$ which means the the coefficient of $T$ in $B_S$ is $\ge t$.  

Since $D$ is contracted to a point on $Z$, we deduce that $T$ is contracted to a point on $R$. 
Applying Theorem \ref{t-general-lcy-dim-2} to $(S,B_S)\to R$, we deduce that 
the gonality of $T$ is bounded depending only on $t$, 
and in case $t>\frac{1}{2}$, the genus of the normalization of $T$ is bounded.\\

\emph{Step 4}.
It remains to treat the case $\dim T=0$. By subsection \ref{construction_mmp} applied to $(Y,B_Y), D$, 
there exist 
\begin{itemize}
\item a log smooth 
pair $(V,\Gamma)$ over $Y$ where every prime exceptional divisor of $V\to Y$ is a component of $\rddown{\Gamma}$, 
and the centre of $D$ on $V$ is a divisor (so it is the birational transform of $D$) hence this centre is a component of $\rddown{\Gamma}$, 
\item an MMP 
$$
V=V_1\bir V_2\bir \cdots \bir V_N=Y
$$ 
on $K_V+\Gamma$  over $Y$,
\item $0<u<t$ depending only on $t$ such that for each $k$,
$$
a(D,V_k,\Gamma_k)\le a(D,Y,(1-\epsilon)B_Y)< 1-u
$$  
where $\Gamma_k$ is the pushdown of $\Gamma$ and $0<\epsilon \ll 1$, and 
\item  in addition, if $t>1/2$, we can assume $u>1/2$. 
\end{itemize}
 
We will denote the centre of $D$ on $V_k$ by $T_k$. Then $T_1$ is a divisor, the birational transform of $D$, 
which is a component of $\rddown{\Gamma_1}$ and $T=T_N$ is a point. 

The steps of the MMP are divisorial contractions or flips. In step $k$, $T_k$ may or may not be 
contained in the indeterminancy locus of $V_k\bir V_{k+1}$, and if it is contained in this locus 
there are several possibilities for the induced map $T_k\to T_{k+1}$, e.g. it maybe contraction to a point 
or it maybe a finite morphism. For the rest of the proof we try to understand these possibilities.\\  
 
\emph{Step 5.}
Assume that $\dim T_k=2$ and $\dim T_{k+1}=0$, for some $k$. Then $T_k$ is a component of ${\Gamma_k}$ with coefficient $1$ 
and $-(K_{V_k}+\Gamma_k)$ is ample over $V_{k+1}$, so applying adjunction we deduce that 
$T_k$ is of Fano type hence it is a rationally connected surface, so $D$ is rational in which case 
we take $C=\mathbb{P}^1$. 

We can then assume that $\dim T_k=1$ for some $k$ because $T_1$ being a surface, we can assume that 
it is contracted to a curve at some step in view of the previous paragraph. 
Let $l$ be the largest number such that $\dim T_l\ge 1$, $V_l \dashrightarrow V_{l+1}$ is a divisorial contraction, say 
contracting a divisor $R$, and $T_l\subset R$. 

We claim that $T_{l+1}$ is either a point or a rational curve. 
This is obvious if $\dim T_{l+1}=0$, so we can assume $\dim T_{l+1}=1$.
Then $V_{l+1}\dasharrow V_N$ is not an isomorphism near the generic point of $T_{l+1}$ because 
$T_{l+1}$ is a curve while $T_N$ is a point. There is 
$k\ge l+1$ such that $V_{l+1} \dashrightarrow V_{k}$ is an isomorphism 
near the generic point of $T_{l+1}$ but $V_{l+1} \dashrightarrow V_{k+1}$ is not. 
In particular, this means that $T_k$ is contained in the indeterminancy locus of $V_k\bir V_{k+1}$. 
Thus, by our choice of $l$, $V_k \dashrightarrow V_{k+1}$ cannot be a divisorial contraction because 
$\dim T_k=1$. Therefore, $V_k \dashrightarrow V_{k+1}$ is a flip and $T_{k}$ is a component of the exceptional 
locus of the corresponding flipping contraction. But then $T_{k}$ is a rational curve by a 
variant of the cone theorem \cite{Sh96}, 
or by \cite{HM07} noting that $T_k$ is a component of some fibre of the flipping contraction, 
hence $T_{l+1}$ is also a rational curve because $V_{l+1}\bir V_k$ 
induces a birational map $T_{l+1}\bir T_k$.\\

\emph{Step 6.}
Assume that $\dim T_l=2$. Then $T_l=R$ is the exceptional divisor of $V_l\bir V_{l+1}$ and it is a 
component of $\rddown{\Gamma_l}$. By the previous step, $T_{l+1}$ is either a point or a rational curve. 
If it is a point, then $T_l$ is rationally connected 
by the previous step and we can take $C=\PP^1$. If it is a curve, then the general fibres of 
$T_l\to T_{l+1}$ are $\mathbb{P}^1$ as $-(K_{V_l}+\Gamma_l)$ is ample over $V_{l+1}$ which implies 
$T_l$ is of Fano type over $T_{l+1}$, 
hence again $T_l$ is rationally connected and we are done by taking $C=\PP^1$ (note that here we are 
using the fact that $T_l\to T_{l+1}$ is a contraction which follows from the facts that $V_l\to V_{l+1}$ 
is a contraction and that $T_l$ is the exceptional locus). 
\\

\emph{Step 7.}
We then assume that $\dim T_l=1$.   
We claim that the gonality of $T_l$ is bounded, and if $t>\frac{1}{2}$, the genus of 
the normalization of $T_l$ is bounded. 
By assumption, $T_l$ is contained in the exceptional divisor $R$ which is a component of $\Gamma_l$ with coefficient $1$. 
Since $-(K_{V_l}+\Gamma_l)$ is ample over $V_{l+1}$, there is $\Theta\ge \Gamma_l$ such that $(V_l,\Theta)$ 
is dlt and $K_{V_l}+\Theta\sim_\R 0/V_{l+1}$. Let $M$ be the normalization of the image of 
$R$ in $V_{l+1}$. Then $R\to M$ is a contraction.  
Applying adjunction 
$$
K_R+\Theta_R:=(K_{V_l}+\Theta)|_R\sim_\R 0 /M,
$$ 
so $(R,\Theta_R)\to M$ is a log Calabi-Yau fibration. 
Moreover, since 
$$
a(D,V_l,\Theta)\le a(D,V_l,\Gamma_l)< 1-u,
$$ 
the coefficient of $T_l$ in $\Theta_R$ 
is $> u$ (this can be seen similar to Step 3 by extracting $D$ via an extremal contraction $V_l'\to V_l$ 
and by restricting to the birational transform of $R$ on $V_l'$). 
Now if $T_l$ is contracted to a point on $M$, then we can apply Theorem \ref{t-general-lcy-dim-2} 
to deduce that the gonality of $T_l$ is bounded, and in case $u>\frac{1}{2}$, the genus of 
the normalization of $T_l$ is bounded. 
\\

\emph{Step 8.}
It remains to treat the case $\dim T_l=\dim T_{l+1}=1$. In this  
case, $T_{l+1}$ is the image of $R$ in $V_{l+1}$, hence 
$M$ is the normalization of $T_{l+1}$.
Since $(R,\Theta_R)\to M$ is a log Calabi-Yau fibration and since $T_l$ is horizontal over $T_{l+1}$, 
taking the intersection number of $\Theta_R>uT_l$ with the general fibres $G$ of $R\to M$ we see that 
$uT_l\cdot G<2$, so the degree of $T_l\to M$ is bounded, hence the gonality of $T_l$ is bounded as claimed.

From now on we assume $t>\frac{1}{2}$ in which case $u>\frac{1}{2}$. 
We prove that in this case the morphism $T_l\to T_{l+1}$ is in fact birational. 
Since $T_{l+1}$ is a curve but $T_N$ is a point, $T_{l+1}$ is contained in the exceptional locus of 
$V_{l+1}\to  Y$. Moreover, since $Y$ is $\Q$-factorial, the exceptional locus is the union of the 
prime exceptional divisors, hence  
$T_{l+1}$ is contained in some exceptional prime divisor of $V_{l+1}\to  Y$, say $S$.  
By construction of the MMP in Step 4, $\rho(V_l/V_{l+1})=1$ and $S$ is a component of $\rddown{\Gamma_{l+1}}$. 
Let $S'$ on $V_l$ be the birational transform of $S$. 
It follows that $S'$ is ample over $V_{l+1}$ and hence it intersects each fibre of the contraction $V_l\to V_{l+1}$.
Then $S'|_R$ gives a component $J$ of $\Theta_R$ with coefficient $1$ which maps onto $T_{l+1}$. 

Assume $J\neq T_l$. Noting that $\Theta_R\ge J+uT_l$ and that $u > \frac{1}{2}$, and then 
considering the intersection numbers 
$$
2=-K_R\cdot G=\Theta_R\cdot G\ge (J+uT_l)\cdot G
$$ 
where $G$ is a general fibre of $R\to M$, we see that 
$T_l\cdot G=1$ as $J\cdot G\ge 1$. 
We conclude that the degree of $T_l\to T_{l+1}$ is $1$. In other words, $T_l\to T_{l+1}$ is birational, and hence $T_l$ is
a rational curve as $T_{l+1}$ is a rational curve.

Now assume $J=T_l$. Applying adjunction $K_R+\Gamma_R:=(K_{V_l}+\Gamma_l)|_R$, the coefficient of 
$T_l$ in $\Gamma_R$ is $1$ because $\Gamma_l\ge R+S'$ and $J=T_l$ is a component of $S'|_R$. Since $-(K_{V_l}+\Gamma_l)$ is ample over $V_{l+1}$, 
we see that 
$$
-2+T_l\cdot G=(K_R+T_l)\cdot G\le (K_R+\Gamma_R)\cdot G<0,
$$ 
hence 
$T_l\cdot G=1$ which again means that $T_l\to T_{l+1}$ is birational and that $T_l$ is a rational curve.

\end{proof}

\begin{proof}(of Theorem \ref{the_theorem-general})
Since $(X,B)$ is dlt and $-(K_X+B)$ is ample over $Z$,
there is $\Delta\ge B$ such that $(X,\Delta)\to Z$ is a log Calabi-Yau fibration. 
In view of Theorem \ref{the_theorem-general-lcy}, it is enough to show that 
$X$ is of Fano type over $Z$. 

Since $(X,B)$ is dlt, by definition, there is a non-empty open subset $U\subset X$ such that 
$(U,B|_U)$ is log smooth and $U$ contains all the non-klt centres of $(X,B)$. Pick a Cartier 
divisor $A$ such that $\mathcal{O}_X(A)$ is an ample invertible sheaf on 
 $X$ (not just relatively over $Z$). Then $\mathcal{O}_X(\rddown{B}+mA)$ is generated 
by global sections for $m\gg 0$. In particular, if $m$ is large enough, then 
for each $x\in U$, there is a section of $\mathcal{O}_X(\rddown{B}+mA)$ not vanishing at $x$ 
because $X$ is smooth at $x$ so $\rddown{B}+mA$ is Cartier at $x$. 
Then the zero divisor of this section gives a divisor $0\le N\sim \rddown{B}+mA$ not passing through 
$x$. Therefore, there is a general $0\le M\sim \rddown{B}+mA$ not containing any non-klt 
centre of $(X,B)$. 

With $A,m,M$ as above, let $\epsilon>0$ be a sufficiently small number, and let 
$$
\Gamma=B-\epsilon \rddown{B}+\epsilon M.
$$ 
Then $(X,\Gamma)$ is klt as can be seen by taking a 
log resolution: indeed, let $W\to X$ be a log resolution and let $K_W+B_W,L_W$ be the pullback of 
$K_X+B, -\rddown{B}+ M$, respectively. Then the pullback of $K_X+\Gamma$ is 
$K_W+B_W+\epsilon L_W$. Since $\epsilon$ is small, any component $D$ of $B_W+\epsilon L_W$ 
with coefficient $\ge 1$ is a component of $B_W$ with coefficient $1$. Thus the centre of $D$ on $X$, say $V$,  
is a non-klt centre of $(X,B)$, so it is not contained in $M$. But then $V|_U$ should be a non-klt centre of 
$(U,(B-\epsilon \rddown{B})|_U)$ which is not possible.

On the other hand, we can make sure that 
$$
-(K_X+\Gamma)=-(K_X+B-\epsilon \rddown{B}+\epsilon M)\sim_\Q -(K_X+B+\epsilon mA)
$$ 
is ample over $Z$ by taking $\epsilon$ to be small enough. Therefore, $X$ is of Fano type over $Z$ 
as desired.

\end{proof}

\begin{proof}(of Theorem \ref{the_theorem})
This follows from Theorem \ref{the_theorem-general}.
\end{proof}



\


\begin{thebibliography}{AF2003}

\bibitem[BPELU17]{BPELU17}
F. Bastianelli, P. De Poi, L. Ein, R. Lazarsfeld, B. Ullery.
\newblock {\em Measures of irrationality for hypersurfaces of large degree}.
\newblock {Compos. Math., 153, 11 (2017), 2368--2393.}


\bibitem[B18]{B18}
C. Birkar.
\newblock {\em Log Calabi-Yau fibrations}.
\newblock {  arXiv:1811.10709v2.}


\bibitem[B19]{B19}
C. Birkar.
\newblock {\em Anti-pluricanonical systems on Fano varieties}.
\newblock { Ann. of Math. Vol. 190, No. 2 (2019), pp. 345--463.}


\bibitem[B16]{B16} 
C. Birkar; 
\newblock {\em Singularities on the base of a Fano type fibration}. 
\newblock {J. Reine Angew Math., {J. Reine Angew Math.}, 715 (2016), 125-142.}


\bibitem[B12]{B12}
C.~Birkar; \emph{Existence of log canonical flips and a special LMMP},
Pub. Math. IHES., 
\textbf{115} (2012), 325-368.
 
\bibitem[BCHM10]{BCHM10}
C. Birkar, P. Cascini, C. Hacon, J. McKernan.
\newblock {\em Existence of minimal models for varieties of log general type}.
\newblock {Journal of Amer. Math. Soc., 23, 2 (2010), 405--467.}

\bibitem[BCDP19]{BCDP19}
J. Blanc, I. Cheltsov, A. Duncan, Yu. Prokhorov.
\newblock {\em Birational self-maps of threefolds of (un)-bounded genus or gonality}.
\newblock {arXiv:1905.00940.}


\bibitem[HM07]{HM07}
C. Hacon, J. McKernan.
\newblock {\em Shokurov's Rational Connectedness Conjecture}.
\newblock {Duke Math. J., 138, 1 (2007), 119--136.}

\bibitem[HMX14]{HMX14}
C. Hacon, J. McKernan, C. Xu.
\newblock {\em ACC for log canonical thresholds}.
\newblock {Journal of the European Math. Soc., 20, 4 (2018), 865--901.}


\bibitem[H77]{H77}
R. Hartshorne.
\newblock {\em Algebraic geometry.}
\newblock {Springer (1977).}


\bibitem[KM98]{KM98}
J\'a. Koll\'ar, Sh. Mori.
\newblock {\em Birational geometry of algebraic varieties.}
\newblock {Cambridge tracts in mathematics (1998).}


\bibitem[Lo19]{Lo19}
K. Loginov.
\newblock {\em On nonrational fibers of del Pezzo fibrations over curves.}
\newblock {Math. Notes, 106, 6 (2019), 930--939.}


\bibitem[Pr19]{Pr19}
Yu. G. Prokhorov.
\newblock {\em Log-canonical degenerations of del Pezzo surfaces in $\mathbb{Q}$-Gorenstein families}.
\newblock {Kyoto J. Math. 59, 4 (2019), 1041--1073.}


\bibitem[Sh96]{Sh96}
V.V. Shokurov, 
\newblock {\em Anticanonical boundedness for curves,}
\newblock {Appendix to "Diagram method for 3-folds and its application to K\"ahler cone and Picard number of Calabi-Yau 3-folds. I." Higher dimensional complex varieties: Proc. of Intern. Confer. held in Trento, Italy, June 15-24, 1994. (eds. M. Andreatta, Th. Peternell) de Gruyter, 1996, 261--328.}

\bibitem[Sh92]{Sh92}
V. V. Shokurov.
\newblock {\em Threefold log flips}.
\newblock {Izv. Ross. Akad. Nauk Ser. Mat. 56, 1 (1992), 105--203.}


\bibitem[Zh06]{Zh06}
Qi Zhang,
\newblock {\em Rational connectedness of log Q-Fano varieties,}
\newblock Journal f\"ur die reine und angewandte Mathematik (Crelles Journal) 590 (2006), 131--142.
\end{thebibliography}

\def\cprime{$'$} \def\mathbb#1{\mathbf#1}

\Addresses 

\end{document}